\documentclass[11pt]{article}

\usepackage{amssymb,amsmath,epsfig,amsthm,enumerate,tikz}
\usepackage{fullpage,color}
\usetikzlibrary{decorations.pathmorphing}

\newtheorem{theorem}{Theorem}[section]

\newtheorem{corollary}[theorem]{Corollary}

\theoremstyle{definition}
\newtheorem{problem}{Problem}

\newcommand\beq{\begin{equation}}
\newcommand\eeq{\end{equation}}
\newcommand\bce{\begin{center}}
\newcommand\ece{\end{center}}
\newcommand\bea{\begin{eqnarray}}
\newcommand\eea{\end{eqnarray}}
\newcommand\bean{\begin{eqnarray*}}
\newcommand\eean{\end{eqnarray*}}
\newcommand\bmt{\begin{multline*}}
\newcommand\emt{\end{multline*}}
\newcommand\ben{\begin{enumerate}}
\newcommand\een{\end{enumerate}}
\newcommand\bit{\begin{itemize}}
\newcommand\eit{\end{itemize}}
\newcommand\brr{\begin{array}}
\newcommand\err{\end{array}}
\newcommand\bt{\begin{tabular}}
\newcommand\et{\end{tabular}}

\newcommand\ms{\medskip}

\newcommand\D{\mathcal D}
\newcommand\V{\mathcal V}
\DeclareMathOperator\val{val}
\DeclareMathOperator\ins{ins}
\DeclareMathOperator\pea{pea}
\DeclareMathOperator\ph{ph}
\DeclareMathOperator\spea{sp}
\DeclareMathOperator\apea{ap}
\DeclareMathOperator\sval{sval}
\renewcommand\t{\mathbf t}
\newcommand\x{\mathbf x}
\newcommand\p{\mathbf p}
\renewcommand\r{\mathbf r}

\newcommand\h{\mathbf h}

\DeclareMathOperator\pw{\mathbf{pea}}
\DeclareMathOperator\sw{\mathbf{sp}}
\DeclareMathOperator\aw{\mathbf{ap}}
\DeclareMathOperator\swsum{spw}
\DeclareMathOperator\svwsum{svw}
\DeclareMathOperator\awsum{apw}
\DeclareMathOperator\phbf{\mathbf{ph}}
\renewcommand\u{\texttt{u}}
\renewcommand\d{\texttt{d}}
\newcommand\W{\mathcal{W}}
\newcommand\M{\mathcal{M}}
\newcommand\bij{\phi}
\def\down{-- ++(1,-1) circle(1.2pt)}
\def\up{-- ++(1,1) circle(1.2pt)}

\title{Symmetric peaks and symmetric valleys in Dyck paths}

\author{Sergi Elizalde\thanks{Department of Mathematics, Dartmouth College, Hanover, NH 03755. \texttt{sergi.elizalde@dartmouth.edu}}}

\date{}

\begin{document}

\maketitle

\begin{abstract}
The notion of symmetric and asymmetric peaks in Dyck paths was introduced by Fl\'orez and Rodr\'{\i}guez, who counted the total number of such peaks over all Dyck paths of a given length.
In this paper we generalize their results by giving multivariate generating functions that keep track of the number of symmetric peaks and the number of asymmetric peaks, as well as the widths of these peaks.  
We recover a formula of Denise and Simion as a special case of our results.
 
 We also consider the analogous but more intricate notion of symmetric valleys. We find a continued fraction expression for the generating function of Dyck paths with respect to the number of symmetric valleys and the sum of their widths, 
which provides an unexpected connection between symmetric valleys and statistics on ordered rooted trees.
 Finally, we enumerate Dyck paths whose peak or valley heights satisfy certain monotonicity and unimodality conditions, using a common framework to recover some known results,
and relating our questions to the enumeration of certain classes of column-convex polyominoes.
\end{abstract}

\section{Introduction}

A {\em Dyck path} of semilength $n$ is a lattice path with steps $\u=(1,1)$ and $\d=(1,-1)$ that starts at $(0,0)$, ends at $(2n,0)$, and never goes below the $x$-axis. Let $\D_n$ be the set of Dyck paths of semilength $n$, and let $\D=\bigcup_{n\ge0}\D_n$ be the set of all Dyck paths.
It is well-known that $|\D_n|=C_n=\frac{1}{n+1}\binom{2n}{n}$, the $n$-th Catalan number. Denoting by $|D|$ the semilength of $D\in\D$, the corresponding generating function is
\begin{equation}\label{eq:C}C(z)=\sum_{D\in\D}z^{|D|}=\sum_{n\ge0}C_nz^n=\frac{1-\sqrt{1-4z}}{2z}.\end{equation}

There is an extensive literature on the enumeration of Dyck paths with respect to the number of occurrences of certain subpaths and other related statistics, see for example~\cite{denise_two_1995,deutsch_dyck_1999,sapounakis_counting_2007,bacher_dyck_2014}. In this paper we explore a different type of statistic that has been recently introduced by 
Fl\'orez and Rodr\'{\i}guez~\cite{florez_enumerating_2020}, namely the number of symmetric and asymmetric peaks, as well as a new related statistic that we call the number of symmetric valleys. We remark that this notion of symmetry is unrelated to the degree of symmetry statistic defined in~\cite{elizalde_measuring_2020}.

A {\em peak} is an occurrence of $\u\d$, and a {\em valley} is an occurrence of $\d\u$. 
Throughout the paper, occurrences of subpaths always refer to subsequences in consecutive positions.
In~\cite{florez_enumerating_2020}, Fl\'orez and Rodr\'{\i}guez introduce the notion of {\em symmetric peaks}, which can be defined as follows. First, note that every peak can be extended to a unique maximal subsequence of the form $\u^i\d^j$ for $i,j\ge1$, which we call the {\em maximal mountain} of the peak. A peak is {\em symmetric} if its maximal mountain $\u^i\d^j$ satisfies $i=j$, and it is {\em asymmetric} otherwise. 
See Figure~\ref{fig:sympeaks} for an example.

\begin{figure}[h]
\centering
    \begin{tikzpicture}[scale=0.55]
     \draw[dotted](-1,0)--(15,0);
     \draw[dotted](0,-1)--(0,4);
     \fill[red!20!white](0,0)--(2,2)--(3,1);
     \fill[green!20!white](3,1)--(5,3)--(7,1);
     \fill[green!20!white](7,1)--(8,2)--(9,1);
      \fill[red!20!white](9,1)--(10,2)--(12,0);
       \fill[green!20!white](12,0)--(13,1)--(14,0);
     \draw[thick](0,0) circle(1.2pt)\up\up\down\up\up\down\down\up\down\up\down\down\up\down;
     \draw[red, thick] (2,2) circle (.2);
     \draw[green, thick] (5,3) circle (.2);
     \draw[green, thick] (8,2) circle (.2);
     \draw[red, thick] (10,2) circle (.2);
     \draw[green, thick] (13,1) circle (.2);
    \end{tikzpicture}
   \caption{A Dyck path with three symmetric peaks and two asymmetric peaks. The symmetric peaks, circled in green, have weights $2,1,1$, whereas both asymmetric peaks, circled in red, have weight $1$. The maximal mountains of each peak are have been filled in. }\label{fig:sympeaks}
\end{figure}
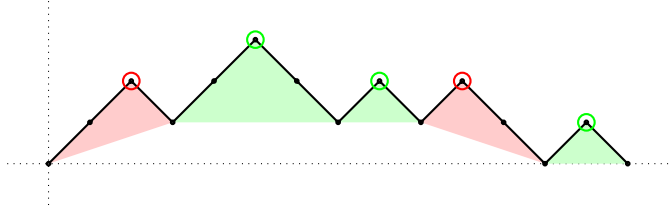

Fl\'orez and Rodr\'{\i}guez~\cite{florez_enumerating_2020}  find a formula for the total number of symmetric peaks over all Dyck paths of semilength $n$, as well as for the total number of asymmetric peaks. 
In~\cite[Sec.\ 2.2]{florez_enumerating_2020}, they pose the more general problem of enumerating Dyck paths of semilength $n$ with a given number of symmetric peaks.
Our first result is a solution to this problem. Specifically, in Section~\ref{sec:peaks} we
obtain a trivariate generating function that enumerates Dyck paths with respect to the number of symmetric peaks and the number of asymmetric peaks. Our method also gives a more direct derivation of the generating function in~\cite[Thm.\ 2.3]{florez_enumerating_2020} for the total number of symmetric peaks.

Fl\'orez and Rodr\'{\i}guez define the {\em weight} of a symmetric peak as the integer $i$ such that the maximal mountain of the peak is $\u^i\d^i$. One can similarly define the weight of an asymmetric peak with maximal mountain $\u^i\d^j$ to be $\min\{i,j\}$.
In~\cite[Cor.\ 3.2]{florez_enumerating_2020}, the authors give a formula for the sum of the weights of all symmetric peaks over all Dyck paths of semilength $n$. We generalize this formula in Section~\ref{sec:peak_weights}, by giving generating functions that enumerate Dyck paths with respect to the number of symmetric peaks and the sum of their weights, and more generally, with respect to the number of symmetric peaks and asymmetric peaks of each possible weight.

Section~\ref{sec:val} deals with the related notion of symmetric valleys, which was originally suggested by Emeric Deutsch~\cite{deustch_personal_nodate}, phrased in terms of pairs of consecutive peaks at the same height.
We say that a valley $\d\u$ is {\em symmetric} if the maximal subsequence of the form $\d^i\u^j$ that contains it satisfies $i=j$; see Figure~\ref{fig:symval} for an example.
The {\em weight} of a symmetric valley is the largest $i$ such that the valley is contained in a subsequence of the form $\d^i\u^i$.
In Section~\ref{sec:valleys} we give a continued fraction expression for the generating function that enumerates Dyck paths with respect to the number of symmetric valleys, and we deduce from it a simple generating function for the total number of symmetric valleys over all Dyck paths of semilength~$n$. In Section~\ref{sec:valley_weights} we incorporate another variable that keeps track of the sum of the weights of the symmetric valleys of a path, and we derive a simple generating function for the total weight of symmetric valleys over all Dyck paths of semilength~$n$. 
For both the total number of symmetric valleys and the total sum of their weights, the resulting formulas can be shown to count other combinatorial structures, providing an intriguing connection between symmetric valleys in Dyck paths and statistics on ordered rooted trees.

\begin{figure}[h]
\centering
    \begin{tikzpicture}[scale=0.55]
     \draw[dotted](-1,0)--(17,0);
     \draw[dotted](0,-1)--(0,4);
     \fill[blue!20!white](7,3)--(9,1)--(11,3);
     \fill[blue!20!white](1,1)--(2,0)--(3,1);
     \draw[thick](0,0) circle(1.2pt)\up\down\up\down\up\up\up\down\down\up\up\down\down\up\down\down;
     \draw[blue, thick] (9,1) circle (.2);
     \draw[blue, thick] (2,0) circle (.2);
    \end{tikzpicture}
   \caption{A Dyck path with two symmetric valleys, circled in blue, having weights $1$ and $2$.}\label{fig:symval}
\end{figure}

Finally, in Section~\ref{sec:unimodal} we use similar ideas to enumerate Dyck paths whose sequence of peak or valley heights satisfies certain monotonicity or unimodality conditions. We recover results by Barcucci et al.~\cite{barcucci_nondecreasing_1997} about so-called non-decreasing Dyck paths, and by Sun and Jia~\cite{sun_counting_2007} and Penaud and Roques~\cite{penaud_generation_2002} about Dyck paths with increasing peak heights. We also obtain new generating functions for Dyck paths whose sequence of peak or valley heights is unimodal, and we show that imposing restrictions on both peak heights and valley heights is sometimes equivalent to the study of certain classes of column-convex polyominoes that have appeared in the literature~\cite{bousquet-melou_generating_1995,bousquet-melou_empilements_1992,flajolet_combinatorial_1980,klarner_asymptotic_1974,stanley_enumerative_2012}.

\section{Symmetric and asymmetric peaks}

\subsection{Counting symmetric and asymmetric peaks}\label{sec:peaks}

For $D\in\D$, let $\pea(D)$, $\spea(D)$, $\apea(D)$ and $\val(D)$ denote the number of peaks, the number of symmetric peaks, the number of asymmetric peaks, and the number of valleys of $D$, respectively. A path in $\D$ is called a {\em pyramid}  if it is of the form $\u^i\d^i$ for some $i\ge1$. 
Our first result gives an expression for the generating function $C_{\spea,\apea}(t,r,z)=\sum_{D\in\D} t^{\spea(D)} r^{\apea(D)} z^{|D|}$.

\begin{theorem}\label{thm:speaapea}
The generating function for Dyck paths with respect to the number of symmetric peaks and the number of asymmetric peaks is
\begin{multline*}C_{\spea,\apea}(t,r,z)\\
=\frac{(1-z)\left(1-tz-(1+t-2r)z^2-\sqrt{(1-z)\left[1-(3+2t)z+(3+4t-4r+t^2)z^2-(1+t-2r)^2 z^3\right]}\right)}{2z(1-(1+t-r)z)^2}.
\end{multline*}
\end{theorem}

\begin{proof}
To obtain an expression for $C_{\spea,\apea}(t,r,z)$, we will describe how to construct Dyck paths where some of their symmetric peaks have been marked, by inserting pyramids with a marked peak in certain positions of smaller Dyck paths.

First, consider the generating function for Dyck paths with respect to the number of valleys, $C_{\val}(v,z)=\sum_{D\in\D}v^{\val(D)}z^{|D|}$. From the standard first-return decomposition of non-empty Dyck paths as $D=\u D_1\d D_2$, where $D_1,D_2\in\D$, one obtains the equation
\begin{equation}\label{eq:Cval} C_{\val}(v,z)=1+zC_{\val}(v,z)(vC_{\val}(v,z)-v+1), \end{equation}
since $\val(D)=\val(D_1)+\val(D_2)+1$ unless $D_2$ is empty. Thus, 
\begin{equation}\label{eq:Cval2}C_{\val}(v,z)=\frac{1-(1-v)z-\sqrt{1-2(1+v)z+(1-v)^2z^2}}{2vz},
\end{equation}
whose coefficients are the well-known Narayana numbers (see e.g.~\cite{deutsch_dyck_1999}).

Next, let $C_{\pea,\ins}(p,q,z)=\sum_{D\in\D}p^{\pea(D)}q^{\ins(D)}z^{|D|}$, where $\ins(D)$
is the number of vertices of $D$ where the insertion of a pyramid $\u^j\d^j$ would create a symmetric peak. Such vertices, which we call {\em insertion points}, are precisely the bottoms of the valleys of $D$, the initial vertex of $D$, and the final vertex of $D$. Thus,  $\ins(D)=\val(D)+2$ and $\pea(D)=\val(D)+1$ unless $D$ is empty, in which case $\ins(D)=1$ and $\pea(D)=0$. It follows that
\begin{equation} C_{\pea,\ins}(p,q,z)=pq^2(C_{\val}(pq,z)-1)+q=\frac{q\left(1+(1-pq)z-\sqrt{1-2(1+pq)z+(1-pq)^2z^2}\right)}{2z}\label{eq:Cpeains}\end{equation}

We are now ready to enumerate Dyck paths with {\em marked} symmetric peaks. Formally, these are pairs $(D,M)$ where $D\in\D$ and $M$ is a subset of the symmetric peaks of $D$. Let $\D^\ast$ be the set of such pairs $(D,M)$, and let
$C_{\pea}^\ast(p,u,z)=\sum_{(D,M)\in\D^{\ast}} p^{\pea(D)} u^{|M|} z^{|D|}$. 
Dyck paths with marked symmetric peaks can be uniquely obtained from regular Dyck paths by inserting, at each insertion point, a possibly empty sequence of pyramids whose peaks are marked. With variable $u$ keeping track of the number of marked peaks,
replacing each insertion point with a sequence of pyramids with marked peaks corresponds to the substitution $q=\frac{1}{1-upz/(1-z)}$ in $C_{\pea,\ins}(p,q,z)$, and so 
$$C_{\pea}^\ast(p,u,z)=C_{\pea,\ins}\left(p,\frac{1}{1-upz/(1-z)},z\right).$$

Our next goal is to get rid of the marking. We introduce a variable $t$ to keep track of the total number of symmetric peaks, not just the marked ones, and we make the substitution $u=t-1$. This way, if $S(D)$ is the set of symmetric peaks of $D\in D$, then
\begin{equation}\label{eq:incl-excl}\sum_{M\subseteq S(D)} (t-1)^{|M|}=((t-1)+1)^{|S(D)|}=t^{\spea(D)},\end{equation}
and so
$$C_{\pea,\spea}(p,t,z)=\sum_{D\in\D} p^{\pea(P)} t^{\spea(P)} z^{|D|}=\sum_{D\in\D}\sum_{M\subseteq S(D)} p^{\pea(D)} (t-1)^{|M|} z^{|D|}=C_{\pea}^\ast(p,t-1,z).$$
Finally, since $\apea(D)=\pea(D)-\spea(D)$, we have $$C_{\spea,\apea}(t,r,z)=C_{\pea,\spea}(r,t/r,z)=C_{\pea,\ins}\left(r,\frac{1}{1-(t-r)z/(1-z)},z\right).$$ Using Equation~\eqref{eq:Cpeains} we obtain the stated formula.
\end{proof}

From Theorem~\ref{thm:speaapea}, one can easily obtain the generating function the total number of symmetric peaks in Dyck paths, recovering~\cite[Thm.\ 2.3]{florez_enumerating_2020}.

\begin{corollary}\label{cor:totalspea}
The generating functions the total number of symmetric peaks and the total number of asymmetric peaks in Dyck paths are, respectively,
$$\sum_{D\in\D} \spea(D) z^{|D|} = \left.\frac{\partial}{\partial t} C_{\spea,\apea}(t,1,z) \right|_{t=1}=\frac{5z-1+(1-z)\sqrt{1-4z}}{2(1-z)\sqrt{1-4z}}=\frac{1}{2}\left(\frac{5z-1}{(1-z)\sqrt{1-4z}}+1\right),$$
$$\sum_{D\in\D} \apea(D) z^{|D|} = \left.\frac{\partial}{\partial r} C_{\spea,\apea}(1,r,z) \right|_{r=1}=\frac{1-3z-(1-z)\sqrt{1-4z}}{(1-z)\sqrt{1-4z}}=
\frac{1-3z}{(1-z)\sqrt{1-4z}}-1.$$
\end{corollary}

An asymptotic analysis of the coefficients of these generating functions also explains the fact, discovered in~\cite[Thm.\ 2.6]{florez_enumerating_2020}, that on the set of Dyck paths of semilength $n$, the proportion of peaks that are symmetric tends to $1/3$ as $n\to\infty$; in other words, Dyck paths have twice as many asymmetric peaks as symmetric peaks on average.

\ms

Another Dyck path statistic that is closely related to the number of symmetric peaks is the number of pairs consecutive valleys at the same height, where the height of a valley is defined to be the $y$-coordinate of its lowest vertex. Let $\spea'(D)$ denote the number of such pairs in $D\in\D$.
This statistic was originally proposed by Emeric Deutsch~\cite{deustch_personal_nodate}.
Note that each pair of consecutive valleys at the same height encloses a symmetric peak. However, the converse is not true: the first and the last peak of the path can be symmetric but they are never enclosed by valleys at the same height, since the endpoints of the path are not valleys. Thus, $\spea'(D)$ and $\spea(D)$ disagree if $D$ starts or ends with a symmetric peak.

It is easy to modify the above argument to obtain a generating function for Dyck paths with respect to this statistic, which we denote by
$C_{\spea'}(t,z)=\sum_{D\in\D} t^{\spea'(D)} z^{|D|}$.

\begin{theorem}\label{thm:spea'}
The generating function for Dyck paths with respect to the number of pairs consecutive valleys at the same height is
$$C_{\spea'}(t,z)=\frac{1+z-tz-\sqrt{(1+z-tz)^2-4z(1-tz)/(1-z)}}{2z}.$$
\end{theorem}

\begin{proof}
The proof is very similar to that of Theorem~\ref{thm:speaapea}, once we remove the variable that keeps track of the total number of peaks. 
The main difference is that the vertices where the insertion of a pyramid $\u^i\d^i$ creates a pair of consecutive valleys at the same height are simply the bottoms of the valleys of the path. Thus,
$$C_{\spea'}(t,z)=C_{\val}\left(\frac{1}{1-(t-1)z/(1-z)},z\right),$$
and the result follows from Equation~\eqref{eq:Cval2}.

An alternative proof is obtained by decomposing non-empty Dyck paths as $D=\u D_1\d\u D_2\d\dots\u D_k\d$, where $k\ge1$ and $D_i\in\D$ for all $i$. Noting that $\spea'(D)$ equals $\sum_{i=1}^k\spea'(D_i)$ plus the number of indices $i\in\{2,3,\dots,k-1\}$ such that $D_i$ is a pyramid or is empty, we get the equation $$C_{\spea'}(t,z)=1+zC_{\spea'}(t,z)+\frac{z^2C_{\spea'}(t,z)^2}{1-z\left(C_{\spea'}(t,z)+\frac{t-1}{1-z}\right)},$$
from where the expression for $C_{\spea'}(t,z)$ follows as well.
\end{proof}

We deduce from Theorem~\ref{thm:spea'} that the generating function for Dyck paths with no pairs of consecutive valleys at the same height is
$$C_{\spea'}(0,z)=\frac{1-z^2-\sqrt{(1-z)(1-3z-z^2-z^3)}}{2z(1-z)}.$$
Interestingly, the same generating function enumerates Dyck paths that avoid the consecutive subpath $\u\u\d\u$, whose coefficients appear as sequence A105633 in~\cite{sloane_-line_nodate}.  This raises the following question.
\begin{problem} Describe a length-preserving bijection between Dyck paths with no pairs of consecutive valleys at the same height and Dyck paths that avoid $\u\u\d\u$.
\end{problem}
Note the distribution of the statistic $\spea'$ on Dyck paths does not coincide with the distribution of the number of occurrences of $\u\u\d\u$.

\subsection{Weights of symmetric and asymmetric peaks}\label{sec:peak_weights} 

Recall that the we defined the {\em weight} of a peak with maximal mountain $\u^i\d^j$ to equal $\min\{i,j\}$.
Fl\'orez and Rodr\'{\i}guez~\cite[Cor.\ 3.2]{florez_enumerating_2020} compute the sum of the weights of all symmetric peaks over all Dyck paths of a given length. In this section we provide a significant refinement of their result and of our formulas from Section~\ref{sec:peaks}, by obtaining a generating function with infinitely many variables that keep track of the number of symmetric and asymmetric peaks of weight $i$, for every $i\ge1$.

For $D\in\D$ and $i\ge1$, let $\pea_i(D)$ (respectively $\spea_i(D)$, $\apea_i(D)$) denote the number of peaks (respectively symmetric peaks, asymmetric peaks) of weight $i$ in $D$. Define the corresponding multivatiate statistics
$\pw(D)=(\pea_1(D),\pea_2(D),\dots)$, 
$\sw(D)=(\spea_1(D),\spea_2(D),\dots)$ and 
$\aw(D)=(\apea_1(D),\apea_2(D),\dots)$. We introduce infinite families of variables $\p=(p_1,p_2,\dots)$, $\t=(t_1,t_2,\dots)$ and $\r=(r_1,r_2,\dots)$, and use the notation $\p^{\pw(D)}=\prod_{i\ge1}p_i^{\pea_i(D)}$, $\t^{\sw(D)}=\prod_{i\ge1}t_i^{\spea_i(D)}$, $\r^{\aw(D)}=\prod_{i\ge1}r_i^{\apea_i(D)}$.
Our next result is an expression for the generating function
$$C_{\sw,\aw}(\t,\r,z)=\sum_{D\in\D} \t^{\sw(D)} \r^{\aw(D)} z^{|D|}.$$ 

\begin{theorem}\label{thm:swaw}
For $\x=(x_1,x_2,\dots)$, let $P(\x,z)=\sum_{i\ge1} x_iz^i$. The generating function for Dyck paths with respect to the weights of their symmetric and asymmetric peaks is
\begin{multline*}C_{\sw,\aw}(\t,\r,z)\\
=\frac{1+z-(1+z)P(\t,z)+2zP(\r,z)-\sqrt{(1-z)\left[(1-P(\t,z))^2-z(1-P(\t,z)+2P(\r,z))^2\right]}}{2z\left(1-P(\t,z)+P(\r,z)\right)^2}.
\end{multline*}
\end{theorem}

\begin{proof}
We modify the proof of Theorem~\ref{thm:speaapea} in order to keep track of the weight of each peak. 
First, let us refine Equation~\eqref{eq:Cval} by introducing variables $p_i$ keeping track of the number of peaks of weight $i$ in the Dyck path, for $i\ge1$. 
In the decomposition of non-empty Dyck paths as $D=\u D_1\d D_2$, peak weights are preserved except when $D_1$ is empty or it is a pyramid, in which case the weight of the corresponding peak increases by one. Thus, letting $C_{\pw,val}(\p,v,z)=\sum_{D\in\D} \p^{\pw(D)}v^{\val(D)}z^{|D|}$, we get
\begin{align}\nonumber &C_{\pw,\val}(\p,v,z)\\ \nonumber
&=1+z\left(C_{\pw,\val}(\p,v,z)+(p_1-1)+(p_2-p_1)z+(p_3-p_2)z^2+\dots\right)(vC_{\pw,\val}(\p,v,z)-v+1)\\
&=1+\left(zC_{\pw,\val}(\p,v,z)-z+(1-z)P(\p,z)\right)(vC_{\pw,\val}(\p,v,z)-v+1).
\label{eq:Cpwval}
\end{align}

Denoting by $\ins(D)$ the number of insertion points of $D$, defined as in the proof of Theorem~\ref{thm:speaapea}, we have
$$C_{\pw,\ins}(\p,q,z)=\sum_{D\in\D}\p^{\pw(D)}q^{\ins(D)}z^{|D|}=q^2(C_{\pw,\val}(\p,q,z)-1)+q.$$
Solving Equation~\eqref{eq:Cpwval} for $C_{\pw,\val}(\p,v,z)$ gives
\begin{equation}C_{\pw,\ins}(\p,q,z)=\frac{q\left(1+z-q(1-z)P(\p,z)-\sqrt{(1-z)\left[(1-qP(\p,z))^2-z(1+qP(\p,z))^2\right]}\right)}{2z}.\label{eq:Cpwins}\end{equation}

Next we obtain a generating function for Dyck paths with marked symmetric peaks, by inserting, at each insertion point, a sequence of pyramids whose peaks are marked. Letting a variable $u_i$ keep track of inserted marked peaks of weight $i$ (while $x_i$ still keeps track of all peaks of weight $i$), such insertions correspond to the substitution $q=1/(1-\sum_{i\ge1}p_iu_iz^i)$ in $C_{\pw,\ins}(\p,q,z)$. A slight variation of Equation~\eqref{eq:incl-excl}, where we replace $S(D)$ with the set of symmetric peaks of weight $i$ of $D$, shows that the substitutions $u_i=t_i-1$, for $i\ge1$, yield the generating function where $t_i$ keeps track of the total number of symmetric peaks of weight $i$ in Dyck paths, that is,
$$C_{\pw,\sw}(\p,\t,z)=\sum_{D\in\D}\p^{\pw(D)}\t^{\sw(D)}z^{|D|}=C_{\pw,\ins}\left(\p,\frac{1}{1-\sum_{i\ge1}t_ip_iz^i+\sum_{i\ge1}p_iz^i},z\right).$$ 

Finally, using the fact that $\apea_i(D)=\pea_i(D)-\spea_i(D)$ for all $i$, the generating function with variables $t_i$ and $r_i$ keeping track of the number of symmetric and asymmetric peaks of weight $i$, respectively, is obtained by making the substitutions $p_i=r_i$ and $t_i=t_i/r_i$, resulting in
$$C_{\sw,\aw}(\t,\r,z)=C_{\pw,\ins}\left(\r,\frac{1}{1-P(\t,z)+P(\r,z)},z\right).$$ 
The stated expression for $C_{\sw,\aw}(\t,\r,z)$ now follows from Equation~\eqref{eq:Cpwins}.
\end{proof}

The first few terms of the series expansion of $C_{\sw,\aw}(\t,\r,z)$ are
$$1+t_1z+(t_1^2+t_2)z^2+(t_1^3+2t_1 t_2+r_1^2+t_3)z^3+(t_1^4+3t_1^2t_2+3t_1r_1^2+2t_1t_3+t_2^2+r_1^2+2r_1r_2+t_4)z^4+\cdots.$$

From Theorem~\ref{thm:swaw}, one can easily deduce a generating function with variables keeping track of the sum of weights of symmetric peaks and the  sum of weights of asymmetric peaks, denoted by $\swsum(P)=\sum_{i\ge1} \sw_i(P)$ and  $\awsum(P)=\sum_{i\ge1} \aw_i(P)$, respectively.
Let $$C_{\spea,\apea,\swsum,\awsum}(t,r,w,y,z)=\sum_{P\in\D} t^{\spea(P)}r^{\apea(P)}w^{\swsum(P)}y^{\awsum(P)} z^{|P|}.$$ 

\begin{corollary}\label{cor:swsum}
The generating function for Dyck paths with respect to the number of symmetric peaks, the number of asymmetric peaks, and the sum of their respective weights is
\begin{multline*}C_{\spea,\apea,\swsum,\awsum}(t,r,w,y,z)\\
=\frac{1+z-\dfrac{twz(1+z)}{1-wz}+\dfrac{2ryz^2}{1-yz}-\sqrt{(1-z)\left[\left(1-\dfrac{twz}{1-wz}\right)^2-z\left(1-\dfrac{twz}{1-wz}+\dfrac{2ryz}{1-yz}\right)^2\right]}}{2z\left(1-\dfrac{twz}{1-wz}+\dfrac{ryz}{1-yz}\right)^2}.
\end{multline*}
\end{corollary}

\begin{proof}
By definition, $C_{\spea,\apea,\swsum,\awsum}(t,r,w,y,z)$ is obtained from $C_{\sw,\aw}(\t,\r,z)$ by making the substitutions $t_i=tw^i$ and $r_i=ry^i$ for all $i\ge1$. When applied to $P(\t,x)$ and $P(\r,x)$, these substitutions yield $twz/(1-wz)$ and $ryz/(1-yz)$, respectively.
The formula now follows from Theorem~\ref{thm:swaw}.
\end{proof}

Note that, setting $w=y=1$ in Corollary~\ref{cor:swsum}, we 
recover the expression from Theorem~\ref{thm:speaapea} for $C_{\spea,\apea}(t,r,z)$. Another interesting specialization of $C_{\spea,\apea,\swsum,\awsum}(t,r,w,y,z)$ is obtained by setting $t=r=1$ and $y=w=q$. In this case, we recover 
the main result from Denise and Simion's paper~\cite{denise_two_1995}, which gives a generating function
where $q$ keeps track of the sum of the weights of all peaks (both symmetric and assymetric) of the path, a statistic referred to as {\em pyramid weight} by Denise and Simion:
$$\sum_{D\in\D}q^{\swsum(D)+\awsum(D)}z^{|D|}=\frac{1+z-2qz+\sqrt{(1-4z)(1-qz)^2+z(1-q)(2+z-3qz)}}{2z(1-qz)}.$$

Finally, setting $t=r=y=1$ in Corollary~\ref{cor:swsum}, differentiating with respect to $w$, and then setting $w=1$, we recover the generating function from~\cite[Thm.\ 3.3]{florez_enumerating_2020}  for the total sum of the weights of symmetric peaks over all paths in $\D_n$.

\begin{corollary}\label{cor:totalswsum}
The generating functions for the total sum of the weights of symmetric peaks 
and the total sum of the weights of asymmetric peaks in Dyck paths are, respectively,
$$\sum_{D\in\D} \swsum(D) z^{|D|} = \left.\frac{\partial}{\partial w} C_{\spea,\apea,\swsum,\awsum}(1,1,w,1,z)\right|_{w=1}=\frac{5z-1+(1-z)\sqrt{1-4z}}{2(1-z)^2\sqrt{1-4z}},$$
$$\sum_{D\in\D} \awsum(D) z^{|D|} = \left.\frac{\partial}{\partial y} C_{\spea,\apea,\swsum,\awsum}(1,1,1,y,z)\right|_{y=1}=\frac{1-3z-(1-z)\sqrt{1-4z}}{(1-z)^2\sqrt{1-4z}}.$$
\end{corollary}

Comparing these formulas with the ones in Corollary~\ref{cor:totalspea}, we see that 
$$\sum_{D\in\D} \swsum(D) z^{|D|}=\frac{1}{1-z}\sum_{D\in\D} \spea(D) z^{|D|},
\qquad\sum_{D\in\D} \awsum(D) z^{|D|}=\frac{1}{1-z}\sum_{D\in\D} \apea(D) z^{|D|}.$$
Equating the coefficients of $z^n$ on both sides, these equalities are equivalent to
\begin{equation}\label{eq:sws_sp} \sum_{D\in\D_n} \swsum(D) = \sum_{k=0}^n\sum_{D\in\D_k}\spea(D),
\qquad 
\sum_{D\in\D_n} \awsum(D) = \sum_{k=0}^n\sum_{D\in\D_k}\apea(D).
\end{equation}
In other words, the sum of the weights of all symmetric peaks of all Dyck paths of semilength $n$ equals the total number of symmetric peaks of all Dyck paths of semilength at most $n$, and similarly for asymmetric peaks. Next we give a simple bijective proof of this phenomenon.

Consider the first equality in Equation~\eqref{eq:sws_sp}; the second equality can be proved analogously.
The right-hand side can be interpreted as counting Dyck paths of semilength at most $n$ with a distinguished symmetric peak. 
The left-hand side can be interpreted as counting pairs $(\hat{D},i)$, where $\hat{D}$ is a Dyck path of semilength $n$ with a distinguished symmetric peak, and $i$ is an integer between $1$ and the weight of this distinguished peak. 

A bijection between the sets counted by both sides can be described as follows. Given a Dyck path of semilength $k\le n$ with a distinguished peak, insert a pyramid $\u^{n-k}\d^{n-k}$ at the top of the distinguished peak. The image is then the pair $(\hat{D},i)$ where $\hat{D}$ is the resulting Dyck path of semilength $n$, with the same distinguished peak, and $i=n-k$. Conversely, given a pair $(\hat{D},i)$ counted by the left-hand side of the equation, delete the steps $\u^i\d^i$ going through the distinguished peak, to obtain a Dyck path of semilength at most $n$ with the same distinguished peak.

\section{Symmetric valleys}\label{sec:val}

\subsection{The number of symmetric valleys}\label{sec:valleys}

In analogy to symmetric peaks, it is natural to consider the notion of symmetric valleys. We say that a valley is symmetric if the maximal subsequence of the form $\d^i\u^j$ that contains it satisfies $i=j$. For a path $D\in\D$, we denote its number of symmetric valleys by $\sval(D)$.
This number coincides with the number of pairs of consecutive peaks at the same height, where the height of a peak is defined to be the $y$-coordinate of its highest vertex. Indeed, symmetric valleys are precisely those enclosed by consecutive peaks at the same height. The study of the number of pairs of consecutive peaks at the same height was originally proposed by Emeric Deutsch~\cite{deustch_personal_nodate}. 

Despite the analogous definition, symmetric valleys in Dyck paths are more difficult to handle than symmetric peaks. The reason is that the method used to insert marked symmetric peaks in the proof of Theorem~\ref{thm:speaapea} does not directly extend to symmetric valleys, since the size of the symmetric valley that we are allowed to insert is bounded by the height of the insertion point.

To address this problem, we will keep track of the heights of the insertion points. The insertion points for marked symmetric valleys are simply the tops of the peaks. For $D\in\D$ and $i\ge1$, let $\ph_i(D)$ denote the number of peaks of $D$ at height $i$, let $\phbf(D)=(\ph_1(D),\ph_2(D),\dots)$, let
$\h^{\phbf(D)}=\prod_{i\ge1} h_i^{\ph_i(D)}$, and let
$$C_{\phbf}(\h,z)=\sum_{D\in\D}\h^{\phbf(D)}z^{|D|}.$$

Every non-empty Dyck path $D\in\D$ can be decomposed uniquely as $D=\u D_1\d\u D_2\d\dots\u D_k\d$, where $k\ge1$ and $D_i\in\D$ for all $i$. The peaks of each $D_i$ become peaks of $D$, with their height increased by one; in addition, $D$ has an extra peak at height $1$ for each empty $D_i$. Iterating this decomposition results in the continued fraction expansion
 \beq\label{eq:Cph}
 C_{\phbf}(\h,z)=\dfrac{1}{1-z(h_1-1)-\dfrac{z}{1-z(h_{2}-1)-\dfrac{z}{1-z(h_{3}-1)-\dfrac{z}{\ddots}}}}.\eeq
Next we use this expression to give a continued fraction for the generating function $C_{\sval}(s,z)=\sum_{D\in\D} s^{\sval(D)}z^{|D|}$. 
For a positive integer $n$ and a variable $q$, we use the $q$-analog notation $$[n]_q=1+q+q^2+\dots+q^{n-1}=\frac{1-q^n}{1-q}.$$

\begin{theorem}\label{thm:sval}
The generating function for Dyck paths with respect to the number of symmetric valleys is
\begin{multline*}
C_{\sval}(s,z)\\
=\dfrac{1}{1+z-\dfrac{z}{1-(s-1)z}-\dfrac{z}{1+z-\dfrac{z}{1-(s-1)(z+z^2)}-\dfrac{z}{1+z-\dfrac{z}{1-(s-1)(z+z^2+z^3)}-\dfrac{z}{\ddots}}}}.
\end{multline*}
\end{theorem}

\begin{proof}
Following a similar argument as in the proofs of Theorems~\ref{thm:speaapea} and~\ref{thm:spea'}, we will construct Dyck paths with marked symmetric valleys by inserting (possibly empty) sequences of {\em pits} $\d^j\u^j$ at the top vertices of the peaks of a Dyck path.
The valleys at the bottom of these inserted pits are marked. The main difference with the previous proofs, however, is that the inserted pits cannot be arbitrarily large: pits inserted at a peak of height $i$ must be of the form $\d^j\u^j$ for $1\le j\le i$.

Such insertions of marked peaks, followed by the inclusion-exclusion argument from Equation~\eqref{eq:incl-excl} that gets rid of the marking, correspond to the substitutions
$$h_i=\frac{1}{1-(s-1)(z+z^2+\dots+z^i)}=\frac{1}{1-(s-1)z[i]_z}$$
in the formula for $C_{\phbf}(\h,z)$ given in Equation~\eqref{eq:Cph}, resulting in the stated expression for $C_{\sval}(s,z)$.
\end{proof}

The first terms of the series expansion of $C_{\sval}(s,z)$ are
$$1+z+(t+1)z^2+(t^2+t+3)z^3+(t^3+t^2+6t+6)z^4+(t^4+t^3+9t^2+15t+16)z^5+\cdots.$$
The constant terms of these polynomials,
$1, 1, 1, 3, 6, 16, 43, 116, 329, 947, 2762, 8176, 24469, \dots$, which enumerate Dyck paths with no symmetric valleys, do not match any existing sequence in~\cite{sloane_-line_nodate}.

Despite not having a closed form for $C_{\sval}(s,z)$, we can use Theorem~\ref{thm:sval} to obtain a simple generating function for the total number of symmetric valleys in Dyck paths.

\begin{corollary}\label{cor:totalsval}
The generating function for the total number of symmetric valleys in Dyck paths is
$$\sum_{D\in\D} \sval(D) z^{|D|} = \left.\frac{\partial}{\partial s} C_{\sval}(s,z) \right|_{s=1}=\frac{2z^2}{1-3z-4z^2+(1-z)\sqrt{1-4z}}.$$
\end{corollary}

\begin{proof}
To compute the partial derivative of $C_{\sval}(s,z)$ with respect to $s$, we define, for $k\ge1$,
$$F_k(s)
=\dfrac{1}{1+z-\dfrac{z}{1-(s-1)z[k]_z}-\dfrac{z}{1+z-\dfrac{z}{1-(s-1)z[k+1]_z}-\dfrac{z}{1+z-\dfrac{z}{1-(s-1)z[k+2]_z}-\dfrac{z}{\ddots}}}}.$$
Note that $F_1(s)=C_{\sval}(s,z)$ by definition, and that $F_k(1)=C_{\sval}(1,z)=C(z)$ for all $k\ge1$, where $C(z)$ is given by Equation~\eqref{eq:C}. Additionally, 
$$F_k(s)=\dfrac{1}{1+z-\dfrac{z}{1-(s-1)z[k]_z}-z F_{k+1}(s)}$$
for $k\ge1$. Differentiating this expression with respect to $s$ and evaluating at $s=1$, we get
$$F'_k(1)=\frac{zF'_{k+1}(1)+z^2[k]_z}{(1-zF_{k+1}(1))^2}=(C(z)-1)(F'_{k+1}(1)+z[k]_z),$$
where we used that $\frac{1}{1-zC(z)}=C(z)$ and $zC(z)^2=C(z)-1$. Iterating this equality, it follows that
$$\left.\frac{\partial}{\partial s} C_{\sval}(s,z) \right|_{s=1}=F'_1(1)=\sum_{i\ge1}(C(z)-1)^i z [i]_z=
\frac{z}{1-z}\left(\frac{C(z)-1}{2-C(z)}-\frac{z(C(z)-1)}{1-z(C(z)-1)}\right),$$
which simplifies to the stated expression.
\end{proof}

The coefficients of the generating function in Corollary~\ref{cor:totalsval} appear as sequence A014301 in~\cite{sloane_-line_nodate}, whose $n$th term counts the total number of internal nodes of even outdegree in all ordered rooted trees with $n$ edges, and also
the total number of protected vertices (i.e., at distance at least $2$ from every leaf) in the same set of trees. 
\begin{problem} 
Find a combinatorial proof of the fact that the total number of symmetric valleys in all paths in $\D_n$ equals the total number of internal nodes of even outdegree in all ordered rooted trees with $n$ edges. 
\end{problem}
Note, however, that the distribution of the number of symmetric valleys over $\D_n$ does not coincide with the distribution of either the number of internal nodes of even outdegree or the number of protected vertices in ordered rooted trees.

\subsection{Weights of symmetric valleys} \label{sec:valley_weights}

Similarly to the definition of peak weights, the {\em weight} of a symmetric valley is the largest $i$ such that the valley is contained a pit $\d^i\u^i$. For $D\in\D$ and $i\ge1$, let $\sval_i(D)$ denote the number of symmetric valleys of weight $i$ in $D$, and let $\svwsum(D)=\sum_{i\ge1} \sval_i(D)$.
It is easy to modify Theorem~\ref{thm:sval} by introducing variables $s_i$ that keep track of the number of symmetric valleys of weight $i$. We will focus on the specialization of this generating function that records the sum of the weights of the symmetric peaks, namely $C_{\sval,\svwsum}(s,w,z)=\sum_{D\in\D}s^{\sval(D)}w^{\svwsum(D)}z^{|D|}$.

\begin{theorem}\label{thm:svw}
The generating function for Dyck paths with respect to the number of symmetric valleys and sum of their weights is
\begin{multline*}
C_{\sval,\svwsum}(s,w,z)\\
=\dfrac{1}{1+z-\dfrac{z}{1+z[1]_z-swz[1]_{wz}}-\dfrac{z}{1+z-\dfrac{z}{1+z[2]_z-swz[2]_{wz}}-\dfrac{z}{1+z-\dfrac{z}{1+z[3]_z-swz[3]_{wz}}-\dfrac{z}{\ddots}}}}.
\end{multline*}
\end{theorem}

\begin{proof}
We modify the proof of Theorem~\ref{thm:sval} in order to keep track of the weight of the inserted marked symmetric valleys. 
The generating function $\sum_{D\in\D}\prod_{i\ge1}s_i^{\sval_i(D)}z^{|D|}$ is obtained by making the substitutions
$$h_i=\frac{1}{1-(s_1-1)z-(s_2-1)z^2-\dots-(s_i-1)z^i}$$
in the continued fraction for $C_{\phbf}(\h,z)$, given in Equation~\eqref{eq:Cph}.

To obtain an expression for $C_{\sval,\svwsum}(s,w,z)$, we make the specialization $s_i=sw^i$, which corresponds to the substitution
$$h_i=\frac{1}{1+z[i]_z-swz[i]_{wz}}$$ in $C_{\phbf}(\h,z)$ for $i\ge1$, resulting in the stated expression.
\end{proof}

\begin{corollary}\label{cor:totalsvwsum}
The generating function for the total sum of the weights of symmetric valleys in Dyck paths is
$$\sum_{D\in\D} \svwsum(D) z^{|D|} = \left.\frac{\partial}{\partial w} C_{\sval,\svwsum}(1,w,z) \right|_{w=1}=\frac{2z^2}{1-3z-3z^2-4z^3+(1-z-3z^2)\sqrt{1-4z}}.$$
\end{corollary}

\begin{proof}
We use a similar argument as in the proof of Corollary~\ref{cor:totalsval}. To compute the partial derivative of $C_{\sval,\svwsum}(1,w,z)$ with respect to $w$, define $G_k(w)$ for $k\ge1$ recursively by 
\begin{equation}\label{eq:Gk} G_k(w)=\dfrac{1}{1+z-\dfrac{z}{1+z[k]_z-wz[k]_{wz}}-z G_{k+1}(w)}. \end{equation}
Then $G_1(w)=C_{\sval,\svwsum}(1,w,z)$, and $G_k(1)=C_{\sval,\svwsum}(1,1,z)=C(z)$ for all $k\ge1$.
Differentiating Equation~\eqref{eq:Gk} with respect to $w$ and evaluating at $w=1$, we get
$$G'_k(1)=\frac{zG'_{k+1}(1)+\dfrac{z^2\left([k]_z-kz^k\right)}{1-z}}{(1-z
G_{k+1}(1))^2}=(C(z)-1)\left(G'_{k+1}(1)+\frac{z\left([k]_z-kz^k\right)}{1-z}\right).$$
Iterating this equality, we get
\begin{multline*}
\left.\frac{\partial}{\partial w} C_{\sval,\svwsum}(1,w,z) \right|_{w=1}=G'_1(1)=\frac{z}{1-z}\sum_{i\ge1}(C(z)-1)^i ([i]_z-iz^i)\\
=\frac{z(C(z)-1)}{1-z}\left(\frac{1}{(1-z)(2-C(z))}-\frac{z}{(1-z)(1-z(C(z)-1))}-\frac{z}{(1-z(C(z)-1))^2}\right),
\end{multline*}
which simplifies to the stated expression.
\end{proof}

The coefficients of the generating function in Corollary~\ref{cor:totalsvwsum} appear as sequence A114515 in~\cite{sloane_-line_nodate}, whose $n$th term counts the total number of peaks in all hill-free (that is, without peaks at height one) Dyck paths of semilength $n$. 
\begin{problem} Find a combinatorial proof of the fact that the total sum of the weights of symmetric valleys in all paths in $\D_n$ equals 
the total number of peaks in all hill-free Dyck paths of semilength~$n$. 
\end{problem}

It would also be interesting to find a common refinement of Theorems~\ref{thm:speaapea} and~\ref{thm:sval}.

\begin{problem} 
Find a generating function for Dyck paths with respect to the number of symmetric peaks and the number of symmetric valleys.
\end{problem}

\section{Increasing and unimodal peak or valley heights}\label{sec:unimodal}

In this section we enumerate Dyck paths whose sequence of peak or valley heights satisfies some monotonicity or unimodality condition. Some of these paths have been studied in the literature: paths whose sequence of valley heights is weakly increasing were introduced by Barcucci et al.~\cite{barcucci_nondecreasing_1997} under the name {\em non-decreasing} Dyck paths, paths whose sequence of peak heights is strictly increasing were counted by Sun and Jia~\cite{sun_counting_2007}, and an expression for paths whose sequence of peak heights is weakly increasing was given by Penaud and Roques~\cite{penaud_generation_2002}.
Here we use a uniform framework to recover these results and to obtain new ones by allowing unimodal sequences of peak or valley heights.

Let $\W^{<}$ (resp.\ $\W^{\le}$) be the set of Dyck paths whose valley heights strictly (resp.\ weakly) increase from left to right. Let $\W^{<>}$ (resp.\ $\W^{\le\ge}$) be the set of Dyck paths whose valley heights are strictly (resp.\ weakly) unimodal. Recall that a sequence $a_1,\dots,a_k$ is strictly (resp.\ weakly) unimodal if there is some $j$ with $1\le j\le k$ such that $a_1<a_2<\dots<a_j>a_{j+1}>\dots>a_k$ (resp.\ $a_1\le a_2\le \dots\le a_j \ge a_{j+1}\ge \dots \ge a_k$).
Define the sets $\M^{<}$, $\M^{\le}$, $\M^{<>}$ and $\M^{\le\ge}$ analogously, with the restrictions placed on the sequence of peak heights instead of the sequence of valley heights.

Denote by $W^{<}(z)$ (resp.\ $M^{<}(z)$) the generating function for paths in $\W^{<}$ (resp.\ $\M^{<}$) with respect to semilength, and similarly for the other sets. In this section we will give expressions for these eight generating functions. 

Let $\Lambda\subset\D$ denote the set of all pyramids, along with the empty path, that is, the set of paths of the form $\u^n\d^n$ for some $n\ge0$. The generating function paths in $\Lambda$ with respect to semilength is $\frac{1}{1-z}$. Note that $\Lambda$ is contained in each of the above eight sets of paths with restricted peak or valley heights.

\subsection{Restrictions on valley heights}\label{sec:valley_heights}

Using the first-return decomposition of Dyck paths, we see that the generating function for Dyck paths with strictly increasing valley heights satisfies
\begin{equation}\label{eq:Av<} W^{<}(z)=1+\frac{z}{1-z}W^{<}(z).\end{equation}
Indeed, any non-empty $D\in\W^{<}$ can be written as $D=P\u D'\d$, where $P\in\Lambda$, and $D'\in\W^{<}$. 

Let us give an alternative derivation of $W^{<}(z)$ that will extend more easily to other cases. Every path $D\in\W^{<}\setminus\Lambda$ must have a valley. Letting $a$ be the height of its highest valley, $D$ can be decomposed uniquely as 
\begin{equation}\label{eq:Av<2decomp}D=P_0\u P_1 \u \dots \u P_{a-1} \u \u P_a \d\u P'_{a} \d^{a+1},\end{equation} 
where $P_i,P'_a\in\Lambda$ for all $i$, and exponentiation denotes repetition. It follows that
\begin{equation}\label{eq:Av<2} W^{<}(z)=\frac{1}{1-z}+ \sum_{a\ge0}\frac{z^{a+2}}{(1-z)^{a+2}}.\end{equation}

Using either of the equations~\eqref{eq:Av<} or~\eqref{eq:Av<2}, we deduce that
$$W^{<}(z)=\frac{1-z}{1-2z},$$
which is also the generating function for compositions. A direct bijection between paths in $\W^{<}$ of semilength $n$ and compositions of $n$ can be described as follows. Let $D\in\W^{<}$. If $D\in\Lambda$, map it to the composition with one part equal to $|D|$. Otherwise, decompose $D$ as in~\eqref{eq:Av<2decomp}, and map it to the composition with $a+2$ parts of sizes $|P_0|+1, |P_1|+1, \dots, |P_{a}|+1,|P'_{a}|+1$.

Similarly, noting that every $D\in\W^{\le}\setminus\Lambda$ whose highest valley is at height $a$ can be decomposed uniquely as 
$$D=S_0\u S_1 \u \dots \u S_{a-1} \u S_a \u P_a \d\u P'_{a} \d^{a+1},$$
where now each $S_i$ is a (possibly empty) sequence of non-empty pyramids, 
and $P_a,P'_{a}\in\Lambda$ as before. It follows that
$$W^{\le}(z)=\frac{1}{1-z}+ \sum_{a\ge0}\frac{z^{a+2}}{(1-z)^2\left(1-\frac{z}{1-z}\right)^{a+1}}=\frac{1-2z}{1-3z+z^2},$$
recovering the generating function for non-decreasing Dyck paths obtained by Barcucci et al.~\cite{barcucci_nondecreasing_1997}.

Adding a variable $t$ to keep track of the number of pairs of consecutive valleys at the same height, the generating function becomes 
$$\sum_{D\in \W^{\le}}t^{\spea'(D)}z^{|D|}= \frac{1}{1-z}+ \sum_{a\ge0}\frac{z^{a+2}}{(1-z)^2\left(1-\frac{tz}{1-z}\right)}
\left(1+\frac{\frac{z}{1-z}}{1-\frac{tz}{1-z}}\right)^{a}=\frac{1-(1+t)z}{1-(2+t)z+tz^2}.$$

For paths whose sequence of valley heights is strictly unimodal, note that every $D\in\W^{<>}\setminus\Lambda$ whose highest valley is at height $a$ can be decomposed uniquely as
$$D=P_0\u P_1 \u \dots \u P_{a-1} \u \u P_a \d\u P'_{a} \d\d P'_{a-1} \d P'_{a-2} \d \dots \d P'_{0},$$
where $P_i,P'_i\in\Lambda$ for all $i$. It follows that
$$W^{<>}(z)=\frac{1}{1-z}+ \sum_{a\ge0}\frac{z^{a+2}}{(1-z)^{2a+2}}=\frac{1-3z+2z^2-z^3}{(1-z)(1-3z+z^2)}.$$
Interestingly, the generating function for paths in $\W^{\le}$ having at most one valley at height $0$ yields the same expression:
$$1+\frac{z}{1-z}W^{\le}(z)=\frac{1-3z+2z^2-z^3}{(1-z)(1-3z+z^2)}.$$ 
The coefficients of this generating function appear as sequence A055588 in~\cite{sloane_-line_nodate}, where an equivalent interpretation in terms of directed column-convex polyominoes is given.

Let us now enumerate paths whose sequence of valley heights is weakly unimodal. Again, every $D\in\W^{\le\ge}\setminus\Lambda$ whose highest valley is at height $a$ can be decomposed uniquely as
$$D=S_0\u S_1 \u \dots \u S_{a-1} \u S_a \u P_a \d\u P'_{a} \d\d S'_{a-1} \d S'_{a-2} \d \dots \d S'_{0},$$
where each $S_i$ and each $S'_i$ is a (possibly empty) sequence of non-empty pyramids, and $P_a,P'_{a}\in\Lambda$.
It follows that
$$W^{\le\ge}(z)=\frac{1}{1-z}+ \sum_{a\ge0}\frac{z^{a+2}}{(1-z)^2\left(1-\frac{z}{1-z}\right)^{2a+1}}=\frac{1-4z+3z^2}{1-5z+6z^2-z^3},$$
which, somewhat surprisingly, coincides with the generating function for Dyck paths of height at most $5$. Its coefficients appear as
A080937 in~\cite{sloane_-line_nodate}. 

Adding a variable $t$ to keep track of the number of pairs of consecutive valleys at the same height, the generating function becomes 
\begin{align*}
\sum_{D\in \W^{\le\ge}}t^{\spea'(D)}z^{|D|}&= 
\frac{1}{1-z}+ \sum_{a\ge0}\frac{z^{a+2}}{(1-z)^2\left(1-\frac{tz}{1-z}\right)}
\left(1+\frac{\frac{z}{1-z}}{1-\frac{tz}{1-z}}\right)^{2a}\\
&=\frac{1-(3+2t)z+(2+4t+t^2)z^2-(1+t+t^2)z^3}{(1-z)\left(1-(2t+3)z+(1+4t+t^2)z^2-t^2 z^3\right)}.
\end{align*}

\subsection{Restrictions on peak heights}\label{sec:peak_heights}

For $i\ge1$, let $\V_i=\{\d^j\u^j:0\le j\le i\}$, that is, the set of pits of semilength at most $i$. 

We start by giving an expression for the generating function $M^{<}(z)$ for Dyck paths with strictly increasing peak heights. 
Every non-empty path $D\in\M^{<}$ whose highest peak is at height $a$ can be decomposed uniquely as
\begin{equation}\label{eq:Ap<decomp}D=\u P_1 \u P_2 \u \dots \u P_{a-1} \u \d^{a},\end{equation} 
where $P_i\in\V_i$ for all $i$. Since the contribution of subpaths in $\V_i$ to the generating function is $1+z+\dots+z^i=[i+1]_z$, it follows that
$$M^{<}(z)=1+\sum_{a\ge1}z^a[2]_z[3]_z\cdots[a]_z=\sum_{a\ge0}z^a[a]_z!,$$
where we use the notation $[a]_z!=[a]_z[a-1]_z\cdots[1]_z$. This expression was found by Sun and Jia~\cite{sun_counting_2007}. The coefficients of its series expansion appear as sequence A008930 in~\cite{sloane_-line_nodate}, whose $n$th term also counts compositions of $n$ whose $i$th part is at most $i$ for all $i$.
A bijection, apparently due to Callan (see the notes in \cite[A008930]{sloane_-line_nodate}),
 from paths $D\in\M^{<}$ of semilength $n$ to such compositions is obtained by letting the first part of the composition be $1$ (assuming $n\ge1$), decomposing $D$ according to~\eqref{eq:Ap<decomp}, and letting the $i$th part of the composition be equal to $|P_{i-1}|+1$ for $2\le i\le a$.

A similar argument yields the  generating function $M^{\le}(z)$ for Dyck paths with weakly increasing peak heights. Every non-empty path  $D\in\M^{\le}$ whose highest peak is at height $a$ can be decomposed as in Equation~\eqref{eq:Ap<decomp}, except that now each $P_i$ consists of a (possibly empty) sequence of non-empty pits from $\V_i$, for all $i$. Since the contribution of such a sequence to the generating function is 
$$\frac{1}{1-(z+z^2+\dots+z^i)}=\frac{1}{1-z[i]_z},$$ we get
$$M^{\le}(z)=\sum_{a\ge0}z^a\prod_{i=1}^{a}\frac{1}{1-z[i]_z}=\sum_{a\ge0}\frac{z^a(1-z)^a}{\prod_{i=1}^{a}(1-2z+z^{i+1})}.$$ 
A different alternating-sum expression for this generating function has been given by Penaud and Roques in~\cite[Theorem 1]{penaud_generation_2002}.
The coefficients of the series expansion of $M^{\le}(z)$ appear as sequence A048285 in~\cite{sloane_-line_nodate}.

Adding a variable $s$ to keep track of the number of pairs of consecutive peaks at the same height (equivalently, symmetric valleys), the generating function becomes 
\begin{align*}
\sum_{D\in \M^{\le}}s^{\sval(D)}z^{|D|}&= 1+\sum_{a\ge1}\frac{z^a}{1-sz[a]_z}\prod_{i=1}^{a-1}\frac{1+(1-s)z[i]_z}{1-sz[i]_z}\\
&=1+\sum_{a\ge1}\frac{z^a(1-z)\prod_{i=1}^{a-1}(1-sz+(s-1)z^{i+1})}{\prod_{i=1}^{a}(1-(s+1)z+sz^{i+1})}.
\end{align*}

For Dyck paths whose peak heights are strictly unimodal, note that every non-empty path $D\in\M^{<>}$ decomposes uniquely as 
$$D=\u P_1 \u P_2 \u \dots \u P_{a-1} \u \d P'_{a-1} \d P'_{a-2} \d \dots \d P'_{1} \d,$$
where $P_i,P'_i\in\V_i$ for all $i$. Thus, we get
$$M^{<>}(z)=\sum_{a\ge0}z^a[a]_z!^2.$$
The first coefficients of its series expansion are $1,1,1,3,6,15,38,95,243,627,1622, 4208$, which dot match any existing sequence in~\cite{sloane_-line_nodate}.

Finally, for Dyck paths with weakly unimodal peak heights, note that every non-empty path $D\in\M^{\le\ge}$ decomposes uniquely as 
$$D=\u S_1 \u S_2 \u \dots \u S_{a-1} \u S_a \d S'_{a-1} \d S'_{a-2} \d \dots \d S'_{1} \d,$$
where each $S_i$ and each $S'_i$ consists of a (possibly empty) sequence of non-empty pits from $\V_i$, for all $i$. It follows that 
$$M^{\le\ge}(z)=1+\sum_{a\ge1}\frac{z^a}{1-z[a]_z}\prod_{i=1}^{a-1}\left(\frac{1}{1-z[i]_z}\right)^2=1+\sum_{a\ge1}\frac{z^a(1-z)^{2a-1}}{\prod_{i=1}^{a-1}(1-2z+z^{i+1})^2}.$$ The first coefficients of its series expansion are $1,1,2,5,14,41,124,383,1200,3796,12088, 38676$, which do not match any existing sequence in~\cite{sloane_-line_nodate}.

Adding a variable $s$ to keep track of the number of pairs of consecutive peaks at the same height, the generating function becomes 
\begin{align*}
\sum_{D\in \M^{\le\ge}}s^{\sval(D)}z^{|D|}&= 1+\sum_{a\ge1}\frac{z^a}{1-sz[a]_z}\left(\prod_{i=1}^{a-1}\frac{1+(1-s)z[i]_z}{1-sz[i]_z}\right)^2\\
&=1+\sum_{a\ge1}\frac{z^a(1-z)}{1-(s+1)z+sz^{a+1}}\left(\prod_{i=1}^{a-1}\frac{1-sz+(s-1)z^{i+1}}{1-(s+1)z+sz^{i+1}}\right)^2.
\end{align*}

\subsection{A connection with column-convex polyominoes}

There are several bijections in the literature between Dyck paths (or subsets of them) and certain classes of column-convex polyominoes. 
Recall that a connected polyomino is {\em column-convex} (resp.\ {\em row-convex}) 
if every column (resp.\ row) 
consists of contiguous cells, and it is {\em convex} if it is both column-convex and row-convex.
A {\em parallelogram polyomino} is a polyomino whose intersection with every line of slope $-1$ is connected. 
A polyomino is {\em directed} if it can be built by starting with a single cell and adding new cells immediately to the right or above an existing cell. 

Delest and Viennot~\cite[Prop.\ 4.1]{delest_algebraic_1984} describe a bijection from $\D_n$ to parallelogram polyominoes 
with semiperimeter $n+1$. Given $D\in\D_n$ with peak heights $a_1,a_2,\dots,a_k$ and valley heights $b_1,b_2,\dots,b_{k-1}$ (for some $k\ge1$), its image is the polyomino consisting of $k$ columns with $a_1,a_2,\dots,a_k$ cells, from left to right, such that the top $b_i+1$ cells of column $i$ are aligned with the bottom $b_i+1$ cells of column $i+1$, for $1\le i\le k$. See Figure~\ref{fig:bij} for an example.
In particular, the height of the $i$th peak of the path equals the number of cells in the $i$th column of the polyomino. Thus, under this bijection, pairs of consecutive peaks at the same height (equivalently, symmetric valleys) correspond to pairs of adjacent columns having the same number of cells. This allows us to interpret our Theorem~\ref{thm:sval} as enumerating parallelogram polyominoes with respect to the number of adjacent columns of the same size. Similarly, the formulas in Section~\ref{sec:peak_heights} can be interpreted as counting parallelogram polyominoes with certain conditions on the column sizes.

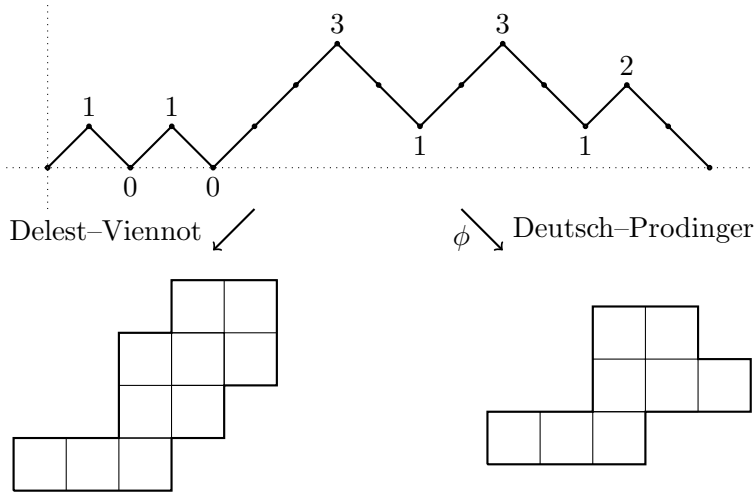
\begin{figure}[htb]
\centering
    \begin{tikzpicture}[scale=0.55]
     \draw[dotted](-1,0)--(17,0);
     \draw[dotted](0,-1)--(0,4);
     \draw[thick](0,0) circle(1.2pt)\up\down\up\down\up\up\up\down\down\up\up\down\down\up\down\down;
     \draw (1,1) node[above]{$1$};
     \draw (3,1) node[above]{$1$};
     \draw (7,3) node[above]{$3$};
     \draw (11,3) node[above]{$3$};
     \draw (14,2) node[above]{$2$};
     \draw (2,0) node[below]{$0$};
     \draw (4,0) node[below]{$0$};
     \draw (9,1) node[below]{$1$};
     \draw (13,1) node[below]{$1$};     
     \draw[thick,->] (5,-1)--(4,-2);
     \draw (4,-1.5) node[left] {Delest--Viennot};
      \draw[thick,->] (10,-1)--(11,-2);
      \draw (11,-1.5) node[right] {Deutsch--Prodinger};
        \draw (10,-1.7) node {$\bij$};
    \end{tikzpicture}\medskip
    
    \begin{tikzpicture}[scale=0.7]
    \draw (0,0) grid (3,1);
     \draw (2,1) grid (4,3);
     \draw (3,2) grid (5,4);
     \draw[thick](0,0)--(3,0)--(3,1)--(4,1)--(4,2)--(5,2)--(5,4)--(3,4)--(3,3)--(2,3)--(2,1)--(0,1)--(0,0);
\begin{scope}[shift={(9,.5)}]
    \draw (0,0) grid (3,1);
     \draw (2,1) grid (4,3);
     \draw[thick](0,0)--(3,0)--(3,1)--(5,1)--(5,2)--(4,2)--(4,3)--(2,3)--(2,1)--(0,1)--(0,0);
\end{scope}
    \end{tikzpicture}
   \caption{The image of a Dyck path  by the Delest--Viennot and the Deustch--Prodinger maps.}\label{fig:bij}
\end{figure}

Deutsch and Prodinger~\cite{deutsch_bijection_2003} describe another bijection between non-decreasing Dyck paths (that is, those whose valley heights are weakly increasing) of semilength $n$ and directed column-convex polyominoes of area $n$. 
Given $D\in\D_n$ with peak heights $a_1,a_2,\dots,a_k$ and valley heights $b_1\le b_2\le \dots\le b_{k-1}$, set $b_0=0$, and define $\bij(D)$ to be the polyomino having $k$ columns, where the $i$th column extends from  $y=b_{i-1}$ to $y=a_i$, for $1\le i\le k$. See Figure~\ref{fig:bij} for an example.
Under this bijection, peak heights of $D$ become the $y$-coordinates of the tops of the columns of $\bij(D)$, whereas valley heights become the $y$-coordinates of the bottoms of the columns. This allows us to combine some of the restrictions from Sections~\ref{sec:valley_heights} and~\ref{sec:peak_heights} on valley heights and peak heights of Dyck paths, and translate them into restrictions on the column boundaries of directed column-convex polyominoes. In some cases, 
the resulting objects are well-studied classes of polyominoes. Here are some examples:

\begin{enumerate}[(a)]
\item When restricted to paths in $\D_n$ with weakly increasing valley heights and weakly increasing peak heights, $\bij$ gives a bijection to parallelogram polyominoes of area $n$. Compare this to the image of Delest and Viennot's bijection, which consists of parallelogram polyominoes with a fixed semiperimeter.
The enumeration of parallelogram polyominoes with respect to their area dates back to work of Klarner and Rivest~\cite{klarner_asymptotic_1974}, see also~\cite[Ex.\ IX.14]{flajolet_combinatorial_1980}. The resulting sequence is recorded as A006958 in~\cite{sloane_-line_nodate}.

\item Restricting $\bij$ to paths in $\D_n$ with weakly increasing valley heights and weakly decreasing peak heights, and then reading the row lengths of the resulting polyomino from top to bottom, we obtain a bijection to weakly unimodal compositions of $n$. Such compositions are enumerated in \cite[Sec.\ 2.5]{stanley_enumerative_2012}; see also~\cite[Ex.\ I.8]{flajolet_combinatorial_1980}, where they are called stack polyominoes. The resulting sequence is A001523 in~\cite{sloane_-line_nodate}.

\item Restricting $\bij$ to paths in $\D_n$ with strictly increasing valley heights and strictly decreasing peak heights, then reading the row lengths of the resulting polyomino from top to bottom, and appending a $1$ at the end, we get a bijection to weakly unimodal compositions of $n+1$ whose adjacent parts differ by at most 1, and whose first and last parts equal 1. These can be enumerated using similar techniques as in case (b); see sequence A001522 in~\cite{sloane_-line_nodate}.

\item Similarly, restricting $\bij$ to paths in $\D_n$ with weakly increasing valley heights and strictly decreasing peak heights, and then reading the row lengths of the resulting polyomino from top to bottom, we obtain weakly unimodal compositions of $n$ where the first part equals one and the adjacent parts in the increasing piece differ by at most 1, which again can be enumerated similarly; see
sequence A001524 in~\cite{sloane_-line_nodate}.

\item When restricted to paths in $\D_n$ with weakly increasing valley heights and weakly unimodal peak heights, $\bij$ gives a bijection to directed convex polyominoes of area $n$. These were first enumerated by Bousquet-M\'elou and Viennot~\cite{bousquet-melou_empilements_1992}, see also~\cite{bousquet-melou_generating_1995} for some history on the enumeration of different classes of convex polyominoes. The resulting sequence is A067676 in~\cite{sloane_-line_nodate}.

\end{enumerate}

\subsection*{Acknowledgments} The author thanks Emeric Deutsch for proposing 
the study of the number of pairs of consecutive peaks (and valleys) at the same height on Dyck paths,
which was in turn inspired by a question of David Scambler in an online forum, and for contributing several ideas.

\bibliographystyle{plain}
\bibliography{peaks_dyck_paths}
\end{document}